\documentclass[reqno]{amsart}
\usepackage{mathrsfs}
\usepackage{color}
\usepackage{amsmath}
\usepackage{amsfonts}
\usepackage{amssymb}
\usepackage{graphicx}

 \newtheorem{Theorem}{Theorem}[section]
 \newtheorem{Corollary}[Theorem]{Corollary}
 \newtheorem{Lemma}[Theorem]{Lemma}

\newtheorem{Question}[Theorem]{Question}

 \newtheorem{Remark}[Theorem]{Remark}

 \numberwithin{equation}{section}

\begin{document}

\title[Modules at boundary points, fiberwise Bergman kernels \uppercase\expandafter{\romannumeral2}]
{Modules at boundary points, fiberwise Bergman kernels, and log-subharmonicity \uppercase\expandafter{\romannumeral2}-- on Stein manifolds}
\author{Shijie Bao}
\address{Institute of Mathematics, Academy of Mathematics
and Systems Science, Chinese Academy of Sciences, Beijing 100190, China.}
\email{bsjie@pku.edu.cn}

\author{Qi'an Guan}
\address{Qi'an Guan: School of
Mathematical Sciences, Peking University, Beijing 100871, China.}
\email{guanqian@math.pku.edu.cn}

\subjclass[2010]{32D15 32E10 32L10 32U05 32W05}

\thanks{}

\keywords{Bergman kernels, $L^2$ extension, strong openness property}

\date{\today}

\dedicatory{}

\commby{}

\maketitle
\begin{abstract}
In this article, we consider Bergman kernels related to modules at boundary points on Stein manifolds and obtain a log-subharmonicity property of the Bergman kernels. 
As applications, we obtain a lower estimate of weighted $L^2$ integrals on Stein manifolds and reprove an effectiveness result of the strong openness property of modules at boundary points on Stein manifolds.
\end{abstract}

\section{Introduction}

The strong openness property of multiplier ideal sheaves (i.e. $\mathcal{I}(\psi)=\mathcal{I}_+(\psi):=\mathop{\cup} \limits_{\epsilon>0}\mathcal{I}((1+\epsilon)\psi)$) has opened the door to new types of approximations, 
which has been widely used in the study of several complex variables, complex algebraic geometry, and complex differential geometry
(see e.g. \cite{GZSOC,K16,cao17,cdM17,FoW18,DEL18,ZZ2018,GZ20,berndtsson20,ZZ2019,ZhouZhu20siu's,FoW20,KS20,DEL21}),
where $\psi$ is a plurisubharmonic function on a complex manifold $M$ (see \cite{Demaillybook}) and the multiplier ideal sheaf $\mathcal{I}(\psi)$ is the sheaf of germs of holomorphic functions $f$ such that $|f|^2e^{-\psi}$ is locally integrable (see e.g. \cite{Tian,Nadel,Siu96,DEL,DK01,DemaillySoc,DP03,Lazarsfeld,Siu05,Siu09,DemaillyAG,Guenancia}).

The strong openness property is a longstanding problem, which was conjectured by Demailly \cite{DemaillySoc,DemaillyAG} and proved by Guan-Zhou \cite{GZSOC} (the 2-dimensional case was proved by Jonsson-Musta\c{t}\u{a} \cite{JM12}). 
Recall that Jonsson and Musta\c{t}\u{a} (see \cite{JM13}, see also \cite{JM12}) posed the following conjecture to prove the strong openness property, and proved the 2-dimensional case \cite{JM12}, which deduces the 2-dimensional strong openness property:

\textbf{Conjecture J-M}: If $c_o^F(\psi)<+\infty$, $\frac{1}{r^2}\mu(\{c_o^F(\psi)\psi-\log|F|<\log r\})$ has a uniform positive lower bound independent of $r\in(0,1)$,
where $\mu$ is the Lebesgue measure on $\mathbb{C}^n$, and $c_o^{F}(\psi):=\sup\{c\geq 0 : |F|^2e^{-2c\psi}$ is locally $L^1$ near $o \}$.

Recall that Guan-Zhou \cite{GZeff} proved Conjecture J-M by using the strong openness property. It is natural to ask: 

\begin{Question}
Can one complete the approach from Conjecture J-M to the strong openness property?
\end{Question}

Bao-Guan-Yuan \cite{BGY} (see also \cite{GMY-BC2}) gave an affirmative answer to the above question by establishing a concavity property of the minimal $L^{2}$ integrals
with respect to a module at a boundary point of the sublevel sets.

In \cite{BG-BB}, we considered Bergman kernels related to modules at boundary points on pseudoconvex domains, and obtained a log-subharmonicity property of the Bergman kernels,
which deduces a new approach from Conjecture J-M to the strong openness property.

In this article, we consider Bergman kernels related to modules at boundary points on Stein manifolds, and obtain the log-subharmonicity property of the Bergman kernels (the case of inner points can be referred to \cite{BG1, BG2}).
As applications, we obtain a lower estimate of weighted $L^2$ integrals on sub-level sets on Stein manifolds, 
and reprove an effectiveness result of strong openness property of modules at boundary points on Stein manifolds.

\subsection{Main result}\label{main}
\
Let $M$ be an $n-$dimensional Stein manifold, and let $K_M$ be the canonical (holomorphic) line bundle on $M$. Let $\psi$ be a plurisubharmonic function on $M$. Let $F\not\equiv 0$ be a holomorphic function on $M$, and let $T\in [-\infty,+\infty)$. Denote that
\[\Psi:=\min\{\psi-2\log |F|,-T\}.\]
If $F(z)=0$ for some $z\in M$, set $\Psi(z)=-T$. For any $t\in [T,+\infty)$, denote that
\[M_t:=\{z\in M : \Psi(z)<-t\}.\]
Note that for any $t\geq T$, $M_t=\{\psi+2\log|1/F|<-t\}$ on $M\setminus\{F=0\}$. Hence $M_t$ is a Stein submanifold of $M$ for any $t\geq T$ (see \cite{FN80}) , and $\Psi=\psi+2\log|1/F|$ is a plurisubharmonic function on $M_t$.

Let $\varphi$ be a plurisubharmonic function on $M$. For any $t\geq T$, denote that
\[A^2(M_t,e^{-\varphi}):=\{f\in H^0(M_t,\mathcal{O}(K_M)) : \int_{M_t}|f|^2e^{-\varphi}<+\infty\},\]
where $|f|^2:=\sqrt{-1}^{n^2}f\wedge\bar{f}$ for any $(n,0)$ form $f$. For any $t\in [T,+\infty)$ and $\lambda>0$, denote that
\[\Psi_{\lambda,t}:=\lambda\max\{\Psi+t,0\}.\]
And for any $f\in A^2(M_T,e^{-\varphi})$, denote that
\[\|f\|_{\lambda,t}:=\left(\int_{M_T}|f|^2e^{-\varphi-\Psi_{\lambda,t}}\right)^{1/2}.\]
Note that
\[\|f\|_T^2:=\|f\|^2_{\lambda,T}=\int_{M_T}|f|^2e^{-\varphi}\]
for any $\lambda>0$, and
\[e^{\lambda(T-t)}\|f\|_T^2\leq\|f\|^2_{\lambda,t}\leq\|f\|_T^2<+\infty\]
for any $t\geq T$.

It is clear that $A^2(M_T,e^{-\varphi})$ is a Hilbert space. Denote the dual space of $A^2(M_T,e^{-\varphi})$ by $A^2(M_T,e^{-\varphi})^*$. For any $\xi\in A^2(M_T,e^{-\varphi})^*$, denote that the Bergman kernel related to $\xi$ is
\[K^{\varphi}_{\xi,\Psi,\lambda}(t):=\sup_{f\in A^2(M_T,e^{-\varphi})}\frac{|\xi\cdot f|^2}{\|f\|^2_{\lambda,t}}\]
for any $t\in [T,+\infty)$, where $K^{\varphi}_{\xi,\Psi,\lambda}(t)=0$ if $A^2(M_T,e^{-\varphi})=\{0\}$.

Denote $E_T:=(T,+\infty)+\sqrt{-1}\mathbb{R}:=\{w\in\mathbb{C} : \text{Re\ }w>T\}\subset\mathbb{C}$. We obtain the following log-subharmonicity property of the Bergman kernel $K^{\varphi}_{\xi,\Psi,\lambda}$.
\begin{Theorem}\label{concavity}
	Assume that $K^{\varphi}_{\xi,\Psi,\lambda}(t_0)\in (0,+\infty)$ for some $t_0\geq T$. Then $\log K^{\varphi}_{\xi,\Psi,\lambda}(\text{Re\ }w)$ is subharmonic with respect to $w\in E_T$.
\end{Theorem}

When $M$ is a pseudoconvex domain in $\mathbb{C}^n$, and $T=0$, Theorem \ref{concavity} can be referred to \cite{BG-BB}.

We recall some notations in \cite{BG-BB,GMY-BC2}. Let $z_0$ be a point in $M$. Denote that
\[\tilde{J}(\Psi)_{z_0}:=\{f\in\mathcal{O}(\{\Psi<-t\}\cap V) : t\in\mathbb{R},\ V \text{\ is \ a \ neighborhood \ of \ } z_0\},\]
and
\[J(\Psi)_{z_0}:=\tilde{J}(\Psi)_{z_0}/\sim,\]
where the equivalence relation `$\sim$' is as follows:
\[f \sim g \ \Leftrightarrow \ f=g \text{\ on \ } \{\Psi<-t\}\cap V, \text{\ where\ }  t\gg 1, V \text{\ is\ a\ neighborhood\ of\ } z_0.\]
For any $f\in \tilde{J}(\Psi)_{z_0}$, denote the equivalence class of $f$ in $J(\Psi)_{z_0}$ by $f_{z_0}$. And for any $f_{z_0},g_{z_0}\in J(\Psi)_{z_0}$, and $(h,z_0)\in\mathcal{O}_{M,z_0}$, define
\[f_{z_0}+g_{z_0}:=(f+g)_{z_0},\ (h,z_0)\cdot f_{z_0}:=(hf)_{z_0}.\]
It is clear that $J(\Psi)_{z_0}$ is an $\mathcal{O}_{M,z_0}-$module. For any $a\geq 0$, denote that $I(a\Psi+\varphi)_{z_0}:=\big\{f_{z_0}\in J(\Psi)_{z_0} : \exists t\gg 1, V \text{\ is\ a\ neighborhood\ of\ } z_0,\ \text{s.t.\ } \int_{\{\Psi<-t\}\cap V}|f|^2e^{-a\Psi-\varphi}dV_M<+\infty\big\}$, where $dV_M$ is a continuous volume form on $M$. Then it is clear that $I(a\Psi+\varphi)_{z_0}$ is an $\mathcal{O}_{M,z_0}-$submodule of $J(\Psi)_{z_0}$. Especially, we denote that $I(\varphi)_{z_0}:=I(0\Psi+\varphi)_{z_0}$, $I(\Psi)_{z_0}:=I(\Psi+0)_{z_0}$, and $I_{z_0}:=I(0\Psi+0)_{z_0}$. Then $I(a\Psi+\varphi)_{z_0}$ is an $\mathcal{O}_{M,z_0}-$submodule of $I(\varphi)_{z_0}$ for any $a>0$. If $z_0\in\bigcap_{t>T}\{\Psi<-t\}$, then $I_{z_0}=\mathcal{O}_{M,z_0}$.

Let $Z_0$ be a subset of $M$. Let $J_{z_0}$ be an $\mathcal{O}_{M,z_0}-$submodule of $J(\Psi)_{z_0}$ for any $z_0\in Z_0$. For any $t\geq T$, denote that
\[A^2(M_t,e^{-\varphi})\cap J:=\left\{f\in A^2(M_t,e^{-\varphi}) :  f_{z_0}\in\mathcal{O}(K_M)_{z_0}\otimes J_{z_0}, \text{for\ any\ } z_0\in Z_0\right\}.\]
Assume that $A^2(M_T,e^{-\varphi})\cap J$ is a proper subspace of $A^2(M_T,e^{-\varphi})$.

Using Theorem \ref{concavity}, we obtain the following concavity property related to $K^{\varphi}_{\xi,\Psi,\lambda}$.

\begin{Theorem}\label{increasing}
	Assume that $I(\Psi+\varphi)_{z_0}\subset J_{z_0}$ for any $z_0\in Z_0$, and assume that $\xi\in A^2(M_T,e^{-\varphi})^*$ such that $\xi|_{A^2(M_T,e^{-\varphi})\cap J}\equiv 0$ and $K^{\varphi}_{\xi,\Psi,\lambda}(t_0)\in (0,+\infty)$ for some $t_0\geq T$. Then $-\log K^{\varphi}_{\xi,\Psi,\lambda}(t)+t$ is concave and increasing with respect to $t\in [T,+\infty)$.
\end{Theorem}

When $M$ is a pseudoconvex domain in $\mathbb{C}^n$, $Z_0$ is one point, and $T=0$, Theorem \ref{increasing} can be referred to \cite{BG-BB}.

\subsection{Applications}
\
Recall that $M$ is an $n-$dimensional Stein manifold. Let $\psi$ be a plurisubharmonic function on $M$. Let $F\not\equiv 0$ be a holomorphic function on $M$, and let $T\in [-\infty,+\infty)$. Let
\[\Psi=\min\{\psi-2\log |F|,-T\}.\]
If $F(z)=0$ for some $z\in M$, set $\Psi(z)=-T$. For any $t\geq T$, denote that $M_t:=\{\Psi<-t\}$. We give the following lower estimate of weighted $L^2$ integrals on sublevel sets $\{\Psi<-t\}$ by Theorem \ref{concavity} and Theorem \ref{increasing}.

\begin{Corollary}[see \cite{GMY-BC2}]\label{L2integral}
	Let $\varphi$ be a plurisubharmonic function on $M$, and let $f$ be a holomorphic $(n,0)$ form on $\{\Psi<-t_0\}$ for some $t_0\geq T$ such that $f\in A^2(M_{t_0},e^{-\varphi})$. Let $z_0\in M$, and assume that $a_{z_0}^f(\Psi;\varphi)<+\infty$, where
	\[a_{z_0}^f(\Psi;\varphi):=\sup\{a\geq 0 : f_{z_0}\in \mathcal{O}(K_M)_{z_0}\otimes I(2a\Psi+\varphi)_{z_0}\}.\]
	Then for any $r\in (0,e^{-a_{z_0}^f(\Psi;\varphi)t_0}]$, we have
	\[\frac{1}{r^2}\int_{\{a_{z_0}^f(\Psi;\varphi)\Psi<\log r\}}|f|^2e^{-\varphi}\geq e^{2a_{z_0}^f(\Psi;\varphi)t_0}C,\]
	where
	\begin{flalign*}
		\begin{split}
			C:=&C(\Psi,\varphi,I_+(2a_{z_0}^f(\Psi;\varphi)\Psi+\varphi)_{z_0},f,M_{t_0})\\
			:=&\inf\bigg\{\int_{M_{t_0}}|\tilde{f}|^2e^{-\varphi} : \tilde{f}\in A^2(M_{t_0},e^{-\varphi})\\ 
			& \& \  (\tilde{f}-f)_{z_0}\in \mathcal{O}(K_M)_{z_0}\otimes I_+(2a_{z_0}^f(\Psi;\varphi)\Psi+\varphi)_{z_0}\bigg\},
		\end{split}
	\end{flalign*} 
	and
	\[I_+(p\Psi+\varphi)_{z_0}=\bigcup_{p'>p}I(p'\Psi+\varphi)_{z_0}\]
	for any $p>0$.
\end{Corollary}

\begin{Remark}\label{a>0}
	In Corollary \ref{L2integral}, for any $z_0\in M$, the proof of $a_{z_0}^f(\Psi;\varphi)>0$ can be referred to \cite{GMY-BC2}.
\end{Remark}

Theorem \ref{concavity} and Theorem \ref{increasing} also deduce a reproof of the following effectiveness result of strong openness property of the module $I(a\Psi+\varphi)_{z_0}$ on Stein manifolds.

\begin{Corollary}[see \cite{GMY-BC2,BG-BB}]\label{SOPE}
	Let $\varphi$ be a plurisubharmonic function on $M$, and let $f$ be a holomorphic $(n,0)$ form on $M_{t_0}=\{\Psi<-t_0\}$ for some $t_0\geq T$ such that $f\in A^2(M_{t_0},e^{-\varphi})$. Let $z_0\in M$, and assume that $a_{z_0}^f(\Psi;\varphi)<+\infty$. Let $C_1$ and $C_2$ be two positive constants. If
	
	(1) $\int_{M_{t_0}}|f|^2e^{-\varphi-\Psi}\leq C_1$;
	
	(2) $C(\Psi,\varphi,I_+(2a_o^f(\Psi;\varphi)\Psi+\varphi)_{z_0},f,M_{t_0})\geq C_2$,\\	
	then for any $q>1$ satisfying
	\[\theta(q)>\frac{C_1}{C_2},\]
	we have $f_{z_0}\in \mathcal{O}(K_M)_{z_0}\otimes I(q\Psi+\varphi)_{z_0}$, where $\theta(q)=\frac{q}{q-1}e^{t_0}$.
\end{Corollary}

\section{Preparations}
\subsection{$L^2$ methods}
\

We recall the optimal $L^2$ extension theorem.

Let $M$ be an $n-$dimensional Stein manifold. Let $D=\Delta_{w_0,r}=\{w\in\mathbb{C} : |w-w_0|<r\}\subset\mathbb{C}$, and $w$ be the coordinate of $D$. Let $\Omega=M\times D$ be an $(n+1)-$dimensional complex manifold, and $p$ be the natural projection from $\Omega$ to $D$. Denote that $\Omega_w:=p^{-1}(w)$ for any $w\in D$. Let $\varphi$ be a plurisubharmonic function on $\Omega$. The following lemma will be used in the proof of Theorem \ref{concavity}.

\begin{Lemma}[Optimal $L^2$ extension theorem (see \cite{guan-zhou12,guan-zhou13ap,guan-zhou13p})]\label{L2ext}
	For any $u$ in $A^2(\Omega_{w_0},e^{-\varphi})$, there exists a holomorphic $(n+1,0)$ form $\tilde{u}$ on $\Omega$, such that $\frac{\tilde{u}}{dw}|_{\Omega_{w_0}}=u$, and
	\[\frac{1}{\pi r^2}\int_{\Omega}\frac{1}{2}|\tilde{u}|^2e^{-\varphi}\leq\int_{\Omega_{w_0}}|u|^2e^{-\varphi}.\]
\end{Lemma}

Let $M$ be an $n-$dimensional weakly pseudoconvex K\"{a}hler manifold. Let $F\not\equiv 0$ be a holomorphic function on $M$, and $\psi$ be a plurisubharmonic function on $M$. Let $\varphi_{\alpha}$ be a Lebesgue measurable function on $M$ such that $\varphi_{\alpha}+\psi$ is a plurisubharmonic function on $M$. Let $T$ be a real number. Denote that
\[\tilde{\varphi}:=\varphi_{\alpha}+2\max\{\psi+T,2\log|F|\},\]
and
\[\Psi:=\min\{\psi-2\log|F|, -T\}.\]
If $F(z)=0$ for some $z\in M$, set $\Psi(z)=-T$. The following lemma will be used to prove Theorem \ref{increasing}.

\begin{Lemma}[\cite{GMY-BC2}]\label{L2mthod}
	Let $t_0\in (T,+\infty)$ be arbitrarily given. Let $f$ be a holomorphic $(n,0)$ form on $\{\Psi<-t_0\}$ such that
	\[\int_{\{\Psi<-t_0\}\cap K}|f|^2<+\infty,\]
	for any compact subset $K\subset M$, and
	\[\int_M\mathbb{I}_{\{-t_0-1<\Psi<-t_0\}}|f|^2e^{-\varphi_{\alpha}-\Psi}<+\infty.\]
	Then there exists a holomorphic $(n,0)$ form $\tilde{F}$ on $M$ such that
	\[\int_M|\tilde{F}-(1-b_{t_0}(\Psi))fF^2|^2e^{-\tilde{\varphi}+v_{t_0}(\Psi)-\Psi}\leq C\int_M\mathbb{I}_{\{-t_0-1<\Psi<-t_0\}}|f|^2e^{-\varphi_{\alpha}-\Psi},\]
	where $b_{t_0}(t)=\int_{-\infty}^t\mathbb{I}_{\{-t_0-1<s<-t_0\}}\mathrm{d}s$, $v_{t_0}(t)=\int_{-t_0}^tb_{t_0}(s)\mathrm{d}s-t_0$ and $C$ is a positive constant.
\end{Lemma}

\subsection{Some lemmas about submodules of $J(\Psi)$}
\
Let $F$ be a holomorphic function on a pseudoconvex domain $D\subset\mathbb{C}^n$ containing the origin $o\in\mathbb{C}^n$. Let $\psi$ be a plurisubharmonic function on $D$. Let $\varphi$ be a plurisubharmonic function on $D$. Let $T$ be a real number. Denote that
\[\Psi:=\min\{\psi-2\log|F|,-T\}.\]
If $F(z)=0$ for some $z\in D$, we set $\Psi(z)=-T$.

We recall the following lemma.

\begin{Lemma}[\cite{GMY-BC2}]\label{l:converge}
	Let $J_o$ be an $\mathcal{O}_{\mathbb{C}^n,o}-$submodule of $I(\varphi)_o$ such that $I(\Psi+\varphi)_o\subset J_o$. Assume that $f_o\in J(\Psi)_o$. Let $U_0$ be a Stein open neighborhood of $o$. Let $\{f_j\}_{j\geq 1}$ be a sequence of holomorphic functions on $U_0\cap\{\Psi<-t_j\}$ for any $j\geq 1$, where $t_j\in (T,+\infty)$. Assume that $t_0=\lim_{j\rightarrow+\infty}t_j\in [T,+\infty)$,
	\[\limsup_{j\rightarrow+\infty}\int_{U_0\cap\{\Psi<-t_j\}}|f_j|^2e^{-\varphi}\leq C<+\infty,\]
	and $(f_j-f)_o\in J_o$. Then there exists a subsequence of $\{f_j\}_{j\geq 1}$ compactly convergent to a holomorphic function $f_0$ on $\{\Psi<-t_0\}\cap U_0$ which satisfies
	\[\int_{U_0\cap\{\Psi<-t_0\}}|f_0|^2e^{-\varphi}\leq C,\]
	and $(f_0-f)_o\in J_o$.
\end{Lemma}

Let $M$ be an $n-$dimensional Stein manifold, and let $K_M$ be the canonical (holomorphic) line bundle on $M$. Let $F\not\equiv 0$ be a holomorphic function on $M$, and $\psi$ be a plurisubharmonic function on $M$. Let $\varphi$ be a plurisubharmonic function on $D$. Let $T\in [-\infty,+\infty)$. Denote that
\[\Psi:=\min\{\psi-2\log|F|, -T\}.\]
If $F(z)=0$ for $z\in M$, set $\Psi(z)=-T$. For any $t\in [T,+\infty)$, denote that
\[M_t:=\{z\in M : \Psi(z)<-t\}.\]
It is well-known that $A^2(M_T,e^{-\varphi})$ is a Hilbert space, where
\[A^2(M_t,e^{-\varphi}):=\{f\in H^0(M_t,\mathcal{O}(K_M)) : \int_{M_t}|f|^2e^{-\varphi}<+\infty\}\]
for any $t\geq T$. Let $Z_0$ be a subset of $M$. Let $J_{z_0}$ be an $\mathcal{O}_{M,z_0}-$submodule of $J(\Psi)_{z_0}$ for any $z_0\in Z_0$. For any $t\geq T$, denote that
\[A^2(M_t,e^{-\varphi})\cap J:=\left\{f\in A^2(M_t,e^{-\varphi}) :  f_{z_0}\in\mathcal{O}(K_M)_{z_0}\otimes J_{z_0}, \text{for\ any\ } z_0\in Z_0\right\}.\]
We state that $A^2(M_T,e^{-\varphi})\cap J$ is a closed subspace of $A^2(M_T,e^{-\varphi})$ if $J_{z_0}\supset I(\Psi+\varphi)_{z_0}$ for any $z_0\in Z_0$.

\begin{Lemma}\label{Jclosed}
	Assume that $J_{z_0}\supset I(\Psi+\varphi)_{z_0}$ for any $z_0\in Z_0$. Then $A^2(M_T,e^{-\varphi})\cap J$ is closed in $A^2(M_T,e^{-\varphi})$.
\end{Lemma}

\begin{proof}
	Let $\{f_j\}$ be a sequence of holomorphic $(n,0)$ forms in $A^2(M_T,e^{-\varphi})\cap J$, such that $\lim_{j\rightarrow+\infty}f_j=f_0$ under the topology of $A^2(M_T,e^{-\varphi})$. Then $\{f_j\}$ compactly converges to $f_0$ on $M_T$. Note that $(f_j)_{z_0}\in \mathcal{O}(K_M)_{z_0}\otimes J_{z_0}$ for any $j$ and $z_0\in Z_0$. According to Lemma \ref{l:converge}, we can get that $(f_0)_{z_0}\in \mathcal{O}(K_M)_{z_0}\otimes J_{z_0}$ for any $z_0\in Z_0$, which means that $f_0\in A^2(M_T,e^{-\varphi})\cap J$. The we know that $A^2(M_T,e^{-\varphi})\cap J$ is closed in $A^2(M_T,e^{-\varphi})$.
\end{proof}

\subsection{Some lemmas about functionals on $A^2(M_T,e^{-\varphi})$}
\

The following two lemmas will be used in the proof of Theorem \ref{concavity}.

\begin{Lemma}\label{fjtof0}
	Let $M$ be an $n-$dimensional complex manifold, and let $\varphi$ be a plurisubharmonic function on $M$. Let $\{f_j\}$ be a sequence in $A^2(M,e^{-\varphi})$, such that $\int_M|f_j|^2e^{-\varphi}$ is uniformly bounded for any $j\in\mathbb{N}_+$. Assume that $f_j$ compactly converges to $f_0\in A^2(M,e^{-\varphi})$. Then for any $\xi\in A^2(M,e^{-\varphi})^*$,
	\[\lim_{j\rightarrow+\infty}\xi\cdot f_j=\xi\cdot f_0.\]
\end{Lemma}

\begin{proof}
	For any $f\in A^2(M,e^{-\varphi})$, denote that $\|f\|^2:=\int_M|f|^2e^{-\varphi}$. Let $\{f_{k_j}\}$ be any subsequence of $\{f_j\}$. Since $A^2(M,e^{-\varphi})$ is a Hilbert space, and $\|f_{k_j}\|^2$ is uniformly bounded, there exists a subsequence of $\{f_{k_j}\}$ (denoted by $\{f_{k_{l_j}}\}$) weakly convergent to some $\tilde{f}\in A^2(M,e^{-\varphi})$. 
	
	Let $\{U_l\}$ be an open cover of the complex manifold $M$, and $(U_l,w_l)$ be the local coordinate on each $U_l$. Then we may denote that $f_j=g_{j,l}dw_l$, $f_0=g_{0,l}dw_l$, and $\tilde{f}=\tilde{g}_ldw_l$ on each $U_l$, where $f_{j,l},g_{0,l}$ and $\tilde{g}_l$ are holomorphic functions on $U_l$. For any $z\in M$, denote that $S_z:=\{l : z\in U_l\}$. And for any $l\in S_z$, let $e_{z,l}$ be the functional defined as follows:
	\begin{flalign*}
		\begin{split}
			e_{z,l} \ : \ A^2(M,e^{-\varphi})&\longrightarrow\mathbb{C}\\
			f&\longmapsto g_l(z),
		\end{split}
	\end{flalign*}
	where $f=g_ldw_l$ on $U_l$, and $g_l$ is a holomorphic function on $U_l$. It is clear that the functional $e_{z,l}\in A^2(D,e^{-\varphi_0})^*$ for any $z\in M$ and any $l\in S_z$. Then we have
	\[f_{0,l}(z)=\lim_{j\rightarrow+\infty}e_{z,l}\cdot f_j=\lim_{j\rightarrow+\infty}e_{z,l}\cdot f_{k_{l_j}}=e_{z,l}\cdot\tilde{f}=\tilde{f}_l(z), \ \forall z\in M, l\in S_z,\]
	thus $f_0=\tilde{f}$. It means that $\{f_{k_j}\}$ has a subsequence weakly convergent to $f_0$. Since $\{f_{k_j}\}$ is an arbitrary subsequence of $\{f_j\}$, we get that $\{f_j\}$ weakly converges to $f_0$. In other words, for any $\xi\in A^2(M,e^{-\varphi})^*$,
	\[\lim_{j\rightarrow+\infty}\xi\cdot f_j=\xi\cdot f_0.\]
\end{proof}

Let $\Omega:=M\times\omega$, where $M$ is an $n-$dimensional complex manifold, and $\omega$ is a domain in $\mathbb{C}$. Let $\varphi_0$ be a plurisubharmonic function on $M$, and $\varphi:=\pi_1^*(\varphi_0)$, where $\pi_1$ is the natural projection from $\Omega$ to $M$. Let $f$ be a holomorphic $(n+1,0)$ form on $\Omega$, such that
\[\int_{\Omega}|f|^2e^{-\varphi}<+\infty.\]
For any $\tau\in\omega$, denote that
\[f_{\tau}:=\frac{f}{d\tau}|_{M_{\tau}}\]
is a holomorphic $(n,0)$ form on $M_{\tau}$, where $M_{\tau}:=\pi_2^{-1}(\tau)$, and $\pi_2$ is the natural projection from $\Omega$ to $\omega$.

\begin{Lemma}\label{xihol}
	For any $\xi\in A^2(M,e^{-\varphi_0})^*$, $\xi\cdot f_{\tau}$ is holomorphic with respect to $\tau\in\omega$.
\end{Lemma}

\begin{proof}
	We only need to prove that $h(\tau):=\xi\cdot f_{\tau}$ is holomorphic near any $\tau_0\in\omega$. Since $\tau_0\in \omega$, there exists $r>0$ such that $\Delta(\tau_0,2r)\subset\subset\omega$. Then for any $\tau\in\Delta(\tau_0,r)$, according to the sub-mean value inequality of subharmonic functions, we have
	\[\int_M|f_{\tau}|^2e^{-\varphi_0}\leq \frac{1}{\pi r^2}\int_{M\times\Delta(\tau,r)}\frac{1}{2}|f|^2e^{-\varphi}\leq\frac{1}{\pi r^2}\int_{\Omega}\frac{1}{2}|f|^2e^{-\varphi}<+\infty,\]
	which implies that $f_{\tau}\in A^2(M,e^{-\varphi_0})$ and there exists $C>0$ such that $\int_M|f_{\tau}|^2e^{-\varphi_0}\leq C$ for any $\tau\in\Delta(\tau_0,r)$.
	
	Let $\{U_l\}$ be an open cover of the complex manifold $M$, and $(U_l,w_l)$ be the local coordinate on each $U_l$.
	For any $z\in M$, Denote that $S_z:=\{l : z\in U_l\}$. And for any $l\in S_z$, let $e_{z,l}$ be the functional in the proof of Lemma \ref{fjtof0}. In the Hilbert space $A^2(M,e^{-\varphi_0})$, by Riesz representation theorem, there exists $\phi_{z,l}\in A^2(M,e^{-\varphi_0})$ such that
	\[e_{z,l}\cdot g=\sqrt{-1}^{n^2}\int_M g\wedge\overline{\phi_{z,l}}e^{-\varphi_0}\]
	for any $z\in M$, $l\in S_z$. Denote that
	\[H:=\overline{\text{span}\{\phi_{z,l} : z\in M, l\in S_z\}}\]
	is a closed subspace of $A^2(M,e^{-\varphi_0})$. If $H\neq A^2(M,e^{-\varphi_0})$, then the closed subspace $H^{\bot}\neq\{0\}$. Choosing some $g_0\in H^{\bot}$ with $g_0\neq 0$, we have that for any $z\in M$ and $l\in S_z$,
	$e_{z,l}\cdot g_0=0$ holds. Then it is clear that $g_0=0$, which is a contradiction. Thus $H=A^2(M,e^{-\varphi_0})$. Denote that
	\[L:=\text{span}\{e_{z,l} : z\in M, l\in S_z\}\subset A^2(M,e^{-\varphi_0})^*.\]
	Since $H=A^2(M,e^{-\varphi_0})$, we can find a sequence $\{\xi_k\}\subset L\subset A^2(M,e^{-\varphi_0})^*$, such that
	\[\lim_{k\rightarrow+\infty}\|\xi_k-\xi\|_{A^2(M,e^{-\varphi_0})^*}=0.\]
	It is clear that for any $z\in M$ and $l\in S_z$, $e_{z,l}\cdot f_{\tau}$ is holomorphic with respect to $\tau\in\omega$. Then for any $k$, $h_k(\tau):=\xi_k\cdot f_{\tau}$ is holomorphic with respect to $\tau\in\omega$.
	Besides, for any $\tau\in\Delta(\tau_0,r)$, we have
	\begin{flalign*}
		\begin{split}
			&|h_k(\tau)-h(\tau)|^2\\
			=&|(\xi_k-\xi)\cdot f_{\tau}|^2\\
			\leq&\|\xi_k-\xi\|^2_{A^2(M,e^{-\varphi_0})^*}\int_M|f_{\tau}|^2e^{-\varphi_0}\\
			\leq&C\|\xi_k-\xi\|^2_{A^2(M,e^{-\varphi_0})^*},
		\end{split}
	\end{flalign*}
	which means that $h_k$ uniformly converges to $h$ on $\Delta(\tau_0,r)$. According to Weierstrass theorem, we know that $h$ is holomorphic on $\Delta(\tau_0,r)$, i.e. near $\tau_0$. Then we get that $\xi\cdot f_{\tau}$ is holomorphic with respect to $\tau\in\omega$.
\end{proof}

\subsection{Some properties of $K^{\varphi}_{\xi,\Psi,\lambda}(t)$}
\

In this section, prove some properties of the Bergman kernel $K^{\varphi}_{\xi,\Psi,\lambda}(t)$.

Let $\xi\in A^2(M_T,e^{-\varphi})^*\setminus\{0\}$. We need the following lemma.

\begin{Lemma}\label{sup=max}
	For any $t\in [T,+\infty)$, if $K^{\varphi}_{\xi,\Psi,\lambda}(t)\in (0,+\infty)$, then there exists $\tilde{f}\in A^2(M_T,e^{-\varphi})$, such that
	\[K^{\varphi}_{\xi,\Psi,\lambda}(t)=\frac{|\xi\cdot \tilde{f}|^2}{\|\tilde{f}\|_{\lambda,t}^2}.\]
\end{Lemma}

\begin{proof}
	By the definition of $K^{\varphi}_{\xi,\Psi,\lambda}(t)$, there exists a sequence $\{f_j\}$ of holomorphic $(n,0)$ forms in $A^2(M_T,e^{-\varphi})$, such that $\|f_j\|_{\lambda,t}=1$, and $\lim_{j\rightarrow+\infty}|\xi\cdot f_j|^2=K^{\varphi}_{\xi,\Psi,\lambda}(t)$. Then $\int_{M_T}|f_j|^2e^{-\varphi}$ is uniformly bounded. Following from Montel's theorem, we can get a subsequence of $\{f_j\}$ compactly convergent to a holomorphic $(n,0)$ form $\tilde{f}$ on $M_T$. According to Fatou's lemma, we have $\|\tilde{f}\|_{\lambda,t}\leq 1$, and according to Lemma \ref{fjtof0}, we have $|\xi\cdot \tilde{f}|^2=K^{\varphi}_{\xi,\Psi,\lambda}(t)$, thus $K^{\varphi}_{\xi,\Psi,\lambda}(t)\leq\frac{|\xi\cdot \tilde{f}|^2}{\|\tilde{f}\|_{\lambda,t}^2}$. Note that $\|\tilde{f}\|_{\lambda,t}\leq 1$ implies $\tilde{f}\in A^2(M_T,e^{-\varphi})$, which means $K^{\varphi}_{\xi,\Psi,\lambda}(t)\geq\frac{|\xi\cdot \tilde{f}|^2}{\|\tilde{f}\|_{\lambda,t}^2}$. Then we get that $K^{\varphi}_{\xi,\Psi,\lambda}(t)=\frac{|\xi\cdot \tilde{f}|^2}{\|\tilde{f}\|_{\lambda,t}^2}$.
\end{proof}

Recall that $Z_0$ is a subset of $M$, and $J_{z_0}$ is an $\mathcal{O}_{M,z_0}-$submodule of $J(\Psi)_{z_0}$ such that $I(\Psi+\varphi)_{z_0}\subset J_{z_0}$ for any $z_0\in Z_0$. For any $t\geq T$, recall that
\[A^2(M_t,e^{-\varphi})\cap J:=\left\{f\in A^2(M_t,e^{-\varphi}) :  f_{z_0}\in\mathcal{O}(K_M)_{z_0}\otimes J_{z_0}, \text{for\ any\ } z_0\in Z_0\right\}.\]
Following from Lemma \ref{Jclosed}, we know that $[A^2(M_T,e^{-\varphi})\cap J$ is a closed subspace of $[A^2(M_t,e^{-\varphi})$. Let $f\in A^2(M_T,e^{-\varphi})$, such that $f\notin A^2(M_T,e^{-\varphi})\cap J$. Recall the minimal $L^2$ integral (\cite{BGY,GMY-BC2}) related to $J$ as follows:
\begin{flalign*}
	\begin{split}
		C(\Psi,\varphi,J,f,M_T):=\inf\bigg\{\int_{M_T}|\tilde{f}|^2e^{-\varphi} :  (\tilde{f}-f)_{z_0}\in (\mathcal{O}(K_M))_{z_0}\otimes J_{z_0}  \text{for\ any\ } z_0\in Z_0&\\
		\& \ \tilde{f}\in H^0(M_T,\mathcal{O}(K_M))&\bigg\}.
	\end{split}
\end{flalign*}
Then the following lemma holds.

\begin{Lemma}\label{B=C}
	Assume that $C(\Psi,\varphi,J,f,M_T)\in (0,+\infty)$, then
	\begin{equation}
		C(\Psi,\varphi,J,f,M_T)=\sup_{\substack{\xi\in A^2(M_T,e^{-\varphi})^*\setminus\{0\}\\ \xi|_{A^2(M_T,e^{-\varphi})\cap J}\equiv 0}}\frac{|\xi\cdot f|^2}{K^{\varphi}_{\xi,\Psi,\lambda}(T)}.
	\end{equation}
\end{Lemma}

\begin{proof}
	Denote that $(\tilde{f}-f)\in J$ if $(\tilde{f}-f)_{z_0}\in (\mathcal{O}(K_M))_{z_0}\otimes J_{z_0}$ for any $z_0\in Z_0$. Note that $\xi\cdot\tilde{f}=\xi\cdot f$ for any $\tilde{f}\in A^2(M_T,e^{-\varphi})$ with $(\tilde{f}-f)\in J$, and $\xi\in A^2(M_T,e^{-\varphi})^*$ satisfying $\xi|_{A^2(M_T,e^{-\varphi})\cap J}\equiv 0$. Then we have
	\begin{flalign*}
		\begin{split}
			K^{\varphi}_{\xi,\Psi,\lambda}(T)&=\sup_{h\in A^2(M_T,e^{-\varphi})}\frac{|\xi\cdot h|^2}{\int_{M_T}|h|^2e^{-\varphi}}\\
			&\geq\sup_{\substack{\tilde{f}\in A^2(M_T,e^{-\varphi})\\ (\tilde{f}-f)\in J}}\frac{|\xi\cdot \tilde{f}|^2}{\int_{M_T}|\tilde{f}|^2e^{-\varphi}}\\
			&= \sup_{\substack{\tilde{f}\in A^2(M_T,e^{-\varphi})\\ (\tilde{f}-f)\in J}}\frac{|\xi\cdot f|^2}{\int_{M_T}|\tilde{f}|^2e^{-\varphi}}.
		\end{split}
	\end{flalign*}
	Thus we get that
	\begin{flalign*}
		\begin{split}
			&\sup_{\substack{\xi\in A^2(M_T,e^{-\varphi})^*\setminus\{0\}\\ \xi|_{A^2(M_T,e^{-\varphi})\cap J}\equiv 0}}\frac{|\xi\cdot f|^2}{K^{\varphi}_{\xi,\Psi,\lambda}(T)}\\
			\leq&\inf_{\substack{\tilde{f}\in A^2(M_T,e^{-\varphi})\\ (\tilde{f}-f)\in J}}\int_{M_T}|\tilde{f}|^2e^{-\varphi}\\
			=&C(\Psi,\varphi,J,f,M_T).
		\end{split}
	\end{flalign*}
	
	Since $A^2(M_T,e^{-\varphi})$ is a Hilbert space, and $A^2(M_T,e^{-\varphi})\cap J$ is a closed proper subspace of $A^2(M_T,e^{-\varphi})$, there exists a closed subspace $H$ of $A^2(M_T,e^{-\varphi})$ such that $H=(A^2(M_T,e^{-\varphi})\cap J)^{\bot}\neq\{0\}$. Then for $f\in A^2(M_T,e^{-\varphi})$, we can make the decomposition $f=f_J+f_H$, such that $f_J\in A^2(M_T,e^{-\varphi})\cap J$, and $f_H\in H$. Note that the linear functional $\xi_f$ defined as follows:
	\[\xi_f\cdot g:=\sqrt{-1}^{n^2}\int_{M_T}g\wedge\overline{f_H}e^{-\varphi}, \ \forall g\in A^2(M_T,e^{-\varphi}),\]
	satisfies that $\xi_f\in A^2(M_T,e^{-\varphi})^*\setminus\{0\}$ and $\xi_f|_{A^2(M_T,e^{-\varphi})\cap J}\equiv 0$. Then we have
	\[\sup_{\substack{\xi\in A^2(M_T,e^{-\varphi})^*\setminus\{0\}\\ \xi|_{A^2(M_T,e^{-\varphi})\cap J}\equiv 0}}\frac{|\xi\cdot f|^2}{K^{\varphi}_{\xi,\Psi,\lambda}(T)}\\
	\geq\frac{|\xi_f\cdot f|^2}{K^{\varphi}_{\xi_f,\Psi,\lambda}(T)}.\]
	Besides, we can know that
	\[K^{\varphi}_{\xi_f,\Psi,\lambda}(T)=\sup_{h\in A^2(M_T,e^{-\varphi})}\frac{|\sqrt{-1}^{n^2}\int_{M_T} h\wedge\overline{f_H}e^{-\varphi}|^2}{\int_{M_T}|h|^2e^{-\varphi}}\leq\int_{M_T}|f_H|^2e^{-\varphi},\]
	and
	\[\xi_f\cdot f=\xi_f\cdot (f_J+f_H)=\xi_f\cdot f_H=\int_{M_T}|f_H|^2e^{-\varphi}.\]
	Then we have
	\[\frac{|\xi_f\cdot f|^2}{K^{\varphi}_{\xi_f,\Psi,\lambda}(T)}\geq\int_{M_T}|f_H|^2e^{-\varphi}\geq C(\Psi,\varphi,J,f,M_T),\]
	which implies that
	\[\sup_{\substack{\xi\in A^2(M_T,e^{-\varphi})^*\setminus\{0\}\\ \xi|_{A^2(M_T,e^{-\varphi})\cap J}\equiv 0}}\frac{|\xi\cdot f|^2}{K^{\varphi}_{\xi,\Psi,\lambda}(T)}\\
	\geq C(\Psi,\varphi,J,f,M_T).\]
	
	Lemma \ref{B=C} is proved.
\end{proof}

Let $\xi\in A^2(M_T,e^{-\varphi})^*$, and recall that the Bergman kernel related to $\xi$ is
\[K^{\varphi}_{\xi,\Psi,\lambda}(t):=\sup_{f\in A^2(M_T,e^{-\varphi})}\frac{|\xi\cdot f|^2}{\|f\|^2_{\lambda,t}}\]
for any $t\in [T,+\infty)$ and $\lambda>0$. We state the following Lemma.

\begin{Lemma}\label{upper-semi}
	$K^{\varphi}_{\xi,\Psi,\lambda}(t)$ is upper-semicontinuous with respect to $t\in[T,+\infty)$, i.e., for any sequence $\{t_j\}_{j=1}^{\infty}$ in $[T,+\infty)$ such that $\lim_{j\rightarrow+\infty}t_j=t_0\in [T,+\infty)$, we have
	\[\limsup_{j\rightarrow+\infty}K^{\varphi}_{\xi,\Psi,\lambda}(t_j)\leq K^{\varphi}_{\xi,\Psi,\lambda}(t_0).\]
\end{Lemma}

\begin{proof}
	
	Denote that
	\[K(t):=K^{\varphi}_{\xi,\Psi,\lambda}(t)\]
	for any $t\in [T,+\infty)$. It can be seen that
	\[e^{\lambda(s-t)}\|f\|^2_{\lambda,s}\leq\|f\|_{\lambda,t}^2\leq\|f\|_{\lambda,s}^2\]
	for any $t>s\geq T$ and $f\in A^2(M_T,e^{-\varphi})$. Note that $K(s)=0$ for some $s\geq T$ induces $K(t)=0$ for any $t\geq T$. Then it suffices to prove Lemma \ref{upper-semi} for $K(t_0)\in (0,+\infty)$ and $K(t_j)\in (0,+\infty)$, $\forall j\in\mathbb{N}_+$.
	
	We assume that $\{t_{k_j}\}$ is the subsequence of $\{t_j\}$ such that
	\[\lim_{j\rightarrow+\infty}K(t_{k_j})=\limsup_{j\rightarrow+\infty}K(t_j).\]
	By Lemma \ref{sup=max}, there exists a sequence of holomorphic $(n,0)$ forms $\{f_j\}$ on $M_T$ such that $f_j\in A^2(M_T,e^{-\varphi})$, $\|f_j\|_{\lambda,t_j}=1$, and $|\xi\cdot f_j|^2=K(t_j)$, for any $j\in\mathbb{N}_+$. Since $\{t_j\}$ is bounded in $\mathbb{C}$, there exists some $s_0>T$, such that $t_j<s_0$ for any $j$, which implies that
	\[\int_{M_T}|f_j|^2e^{-\varphi}\leq e^{\lambda (s_0-T)}\|f_j\|^2_{\lambda,t_j}=e^{\lambda (s_0-T)}, \ \forall j\in\mathbb{N}_+.\]
	Then following from Montel's theorem, we can get a subsequence of $\{f_{k_j}\}$ (denoted by $\{f_{k_j}\}$ itself) compactly convergent to a holomorphic $(n,0)$ form $f_0$ on $M_T$. According to Fatou's lemma, we have
	\begin{flalign*}
		\begin{split}
			\|f_0\|_{\lambda,t_0}&=\int_{M_T}|f_0|^2e^{-\varphi-\lambda\max\{\Psi+t_0,0\}}\\
			&=\int_{M_T}\lim_{j\rightarrow+\infty}|f_{k_j}|^2e^{-\varphi-\lambda\max\{\Psi+t_{k_j},0\}}\\
			&\leq\liminf_{j\rightarrow+\infty}\int_{M_T}|f_{k_j}|^2e^{-\varphi-\lambda\max\{\Psi+t_{k_j},0\}}\\
			&=\liminf_{j\rightarrow+\infty}\|f_{k_j}\|_{\lambda,t_j}=1.
		\end{split}
	\end{flalign*}
	Then $\int_{M_T}|f_0|^2e^{-\varphi}\leq e^{\lambda (t_0-T)} \|f_0\|^2_{\lambda,t_0}\leq e^{\lambda (s_0-T)}<+\infty$, which implies that $f_0\in A^2(M_T,e^{-\varphi})$. Lemma \ref{fjtof0} shows that $|\xi\cdot f_0|^2=\lim_{j\rightarrow+\infty}|\xi\cdot f_{k_j}|^2=\limsup_{j\rightarrow+\infty}K(t_j)$. Thus
	\[K(t_0)\geq\frac{|\xi\cdot f_0|^2}{\|f_0\|^2_{\lambda,t_0}}\geq\limsup_{j\rightarrow+\infty}K(t_j),\]
	which means that $K(t)$ is upper semi-continuous with respect to $t\in [T,+\infty)$.
\end{proof}

\section{Proof of Theorem \ref{concavity}}

We prove Theorem \ref{concavity} by using Lemma \ref{L2ext} (optimal $L^2$ extension theorem).

\begin{proof}[Proof of Theorem \ref{concavity}]
	Denote that $\Omega:=M_T\times E_T$. Note that $M_T$ is an $n-$dimensional Stein manifold, and $\Psi$ is a plurisubharmonic function on $M_T$. We get that $\Omega$ is an $(n+1)-$dimensional Stein manifold. Denote that $\pi_1,\pi_2$ are the natural projections from $\Omega$ to $M_T$ and $E_T$. Let
	\[\tilde{\Psi}:=\varphi(z)+\lambda\max\{\Psi(z)+\text{Re\ }w,0\}\]
	for any $(z,w)\in \Omega$ with $z\in M_T$ and $w\in E_T$. Then $\tilde{\Psi}$ is a plurisubharmonic function on $\Omega$.
	
	Denote that
	\[K(w):=K^{\varphi}_{\xi,\Psi,\lambda}(\text{Re\ }w)\]
	for any $w\in E_T$. We prove that $\log K(w)$ is a subharmonic function with respect to $w\in E_T$.
	
	Firstly we prove that $\log K(w)$ is upper semicontinuous. Let $w_j\in E$ such that $\lim_{\lambda\rightarrow+\infty}w_j=w_0\in E_T$. Then $\lim_{j\rightarrow+\infty}\text{Re\ }w_j=\text{Re\ }w_0\in [T,+\infty)$. Following from Lemma \ref{upper-semi}, we get that
	\[\limsup_{j\rightarrow+\infty}\log K(w_j)\leq \log K(w_0).\]
	Thus $\log K(w)$ is upper semincontinuous with respect to $w\in E_T$.
	
	Secondly we prove that $\log K(w)$ satifies the sub-mean value inequality on $E_T$.
	
	Let $w_0\in E_T$, and $\Delta(w_0,r)\subset E_T$ be the disc centered at $w_0$ with radius $r$. Let $\Omega':=M_T\times\Delta(w_0,r)\subset \Omega$ be a submanifold of $\Omega$. Let $f_0\in A^2(M_T,e^{-\varphi})$ such that
	\[K(w_0)=\frac{|\xi\cdot f_0|^2}{\|f_0\|^2_{\lambda,\text{Re\ }w_0}}\]
	by Lemma \ref{sup=max}.
	
	Note that  $\Omega'$ is a Stein manifold, and $\tilde{\Psi}(z,w)=\varphi(z)+\Psi_{\lambda,\text{Re\ }w}=\varphi(z)+\lambda\max\{\Psi(z)+\text{Re\ }w,0\}$ is a plurisubharmonic function on $\Omega'$. Using Lemma \ref{L2ext} (optimal $L^2$ extension theorem), we can get a holomorphic $(n,0)$ form $\tilde{f}$ on $\Omega'$ such that $g(z,w_0)=g_0(z)$ for any $z\in M_T$, where $\tilde{f}=g(z,w)dz\wedge dw$ and $f_0=g(z)dz$ on $(z,w)$, and
	\begin{equation}\label{ineqL2ext}
		\frac{1}{\pi r^2}\int_{\Omega'}\frac{1}{2}|\tilde{f}|^2e^{-\tilde{\Psi}}\leq \int_{M_T}|f_0|^2e^{-\Psi_{\lambda,\text{Re\ }w_0}}.
	\end{equation}
	Denote that $\tilde{f}_w=\frac{\tilde{f}}{dw}|_{M_T\times\{w\}}$. Since the function $y=\log x$ is concave, according to Jensen's inequality and inequality (\ref{ineqL2ext}), we have
	\begin{flalign}\label{ineqJensen}
		\begin{split}
			\log\|f_0\|^2_{\lambda,\text{Re\ }w_0}&=\log\left(\int_{M_T}|f_0|^2e^{-\Psi_{\lambda,\text{Re\ }w_0}}\right)\\
			&\geq\log\left(\frac{1}{\pi r^2}\int_{\Omega'}\frac{1}{2}|\tilde{f}|^2e^{-\tilde{\Psi}}\right)\\
			&=\log\left(\frac{1}{\pi r^2}\int_{\Delta(w_0,r)}\left(\int_{M_T\times\{w\}}|\tilde{f}_w|^2e^{-\Psi_{\lambda,\text{Re\ }w}}\right)d\mu_{\Delta_{w_0,r}}(w)\right)\\
			&\geq\frac{1}{\pi r^2}\int_{\Delta(w_0,r)}\log\left(\|\tilde{f}_w\|^2_{\lambda,\text{Re\ }w}\right)d\mu_{\Delta_{w_0,r}}(w)\\
			&\geq\frac{1}{\pi r^2}\int_{\Delta(w_0,r)}\left(\log|\xi\cdot \tilde{f}_w|^2-\log K(w)|\right)d\mu_{\Delta_{w_0,r}}(w).
		\end{split}
	\end{flalign}
	Where $\mu_{\Delta_{w_0,r}}$ is the Lebesgue measure on $\Delta_{w_0,r}$. It follows from Lemma \ref{xihol} that $\xi\cdot \tilde{f}_w$ is holomorphic with respect to $w$, which implies that $\log|\xi\cdot\tilde{f}_w|^2$ is subharmonic with respect to $w$. Then we have
	\[\log|\xi\cdot f_0|^2\leq\frac{1}{\pi r^2}\int_{\Delta(w_0,r)}\log|\xi\cdot \tilde{f}_w|^2d\mu_{\Delta_{w_0,r}}(w).\]
	Combining with inequality (\ref{ineqJensen}), we get
	\[\log\|f_0\|^2_{\lambda,\text{Re\ }w_0}\geq\log|\xi\cdot f_0|^2-\frac{1}{\pi r^2}\int_{\Delta(w_0,r)}\log K(w)d\mu_{\Delta_{w_0,r}}(w),\]
	which means
	\[\log K(w_0)\leq\frac{1}{\pi r^2}\int_{\Delta(w_0,r)}\log K(w)d\mu_{\Delta_{w_0,r}}(w).\]
	
	Since $\log K(w)$ is upper semicontinuous and satisfies the sub-mean value inequality on $E_T$, we know that $\log K(w)$ is a subharmonic function on $E_T$.
\end{proof}

\section{Proof of Theorem \ref{increasing}}
In this section, we give the proof of Theorem \ref{increasing}. We need the following lemma.

\begin{Lemma}[see \cite{Demaillybook}]\label{Re}
	Let $D=I+\sqrt{-1}\mathbb{R}:=\{z=x+\sqrt{-1}y\in\mathbb{C} : x\in I, y\in\mathbb{R}\}$ be a subset of $\mathbb{C}$, where $I$ is an interval in $\mathbb{R}$. Let $\phi(z)$ be a subharmonic function on $D$ which is only dependent on $x=\text{Re\ }z$. Then $\phi(x):=\phi(x+\sqrt{-1}\mathbb{R})$ is a convex function with respect to $x\in I$.
\end{Lemma}

\begin{proof}[Proof of Theorem \ref{increasing}]
	It follows from Theorem \ref{concavity} that $\log K^{\varphi}_{\xi,\Psi,\lambda}(\text{Re\ } w)$ is subharmonic with respect to $w\in (T,+\infty)+\sqrt{-1}\mathbb{R}$. Note that $\log K^{\varphi}_{\xi,\Psi,\lambda}(\text{Re\ } w)$ is only dependent on $\text{Re\ }w$, then following from Lemma \ref{Re}, we get that $\log K^{\varphi}_{\xi,\Psi,\lambda}(t)=\log K^{\varphi}_{\xi,\Psi,\lambda}(t+\sqrt{-1}\mathbb{R})$ is convex with respect to $t\in (T,+\infty)$. Combining with Lemma \ref{upper-semi}, we get that $\log K^{\varphi}_{\xi,\Psi,\lambda}(t)$ is convex with respect to $t\in [T,+\infty)$, which implies that $-\log K^{\varphi}_{\xi,\Psi,\lambda}(t)+t$ is concave with respect to $t\in [T,+\infty)$. Then for any $\xi\in A^2(M_T,e^{-\varphi})^*$ with $\xi|_{A^2(M_T,e^{-\varphi})\cap J }\equiv 0$, to prove that $\log -K^{\varphi}_{\xi,\Psi,\lambda}(t)+t$ is increasing, we only need to prove that $\log -K^{\varphi}_{\xi,\Psi,\lambda}(t)+t$ has a lower bound on $[T,+\infty)$.
	
	Using Lemma \ref{sup=max}, we obtain that there exists $f_t\in A^2(M_T,e^{-\varphi})$ for any $t\in [T,+\infty)$, such that $\xi\cdot f_t=1$ and
	\begin{equation}\label{K=ft}
		K^{\varphi}_{\xi,\Psi,\lambda}(t)=\frac{1}{\|f_t\|_{\lambda,t}^2}.
	\end{equation}
	In addition, according to Lemma \ref{L2mthod}, there exists a holomorphic $(n,0)$ form $\tilde{F}$ on $M_T$ such that
	\begin{equation}\label{tildeF}
		\int_{M_T}|\tilde{F}-(1-b_{t}(\Psi))f_tF^2|^2e^{-\tilde{\varphi}+v_{t}(\Psi)-\Psi}\leq C\int_{M_T}\mathbb{I}_{\{-t-1<\Psi<-t\}}|f_t|^2e^{-\tilde{\varphi}-\Psi},
	\end{equation}
	where
	\[\tilde{\varphi}=\varphi+2\max\{\psi+T,2\log|F|\},\]
	and $C$ is a positive constant. Then it follows from inequality (\ref{tildeF}) that
	\begin{flalign}\label{tildeF2}
		\begin{split}
			&\int_{M_T}|\tilde{F}-(1-b_{t}(\Psi))f_tF^2|^2e^{-\tilde{\varphi}+v_{t}(\Psi)-\Psi}\\
			\leq&C\int_{M_T}\mathbb{I}_{\{-t-1<\Psi<-t\}}|f_t|^2e^{-\varphi-\Psi}\\
			\leq&Ce^{t+1}\int_{\{\Psi<-t\}}|f_t|^2e^{-\varphi}.
		\end{split}
	\end{flalign}
	Denote that $\tilde{F}_t:=\tilde{F}/F^2$ on $M_T$, then $\tilde{F}_t$ is a holomorphic $(n,0)$ form on $M_T$. Note that $|F|^4e^{-\tilde{\varphi}}=e^{-\varphi}$ on $M_T$. Then inequality (\ref{tildeF2}) implies that
	\begin{equation}\label{tildeFt}
		\int_{M_T}|\tilde{F}_t-(1-b_{t}(\Psi))f_t|^2e^{-\varphi+v_{t}(\Psi)-\Psi}\leq Ce^{t+1}\int_{\{\Psi<-t\}}|f_t|^2e^{-\varphi}<+\infty.
	\end{equation}
	According to inequality (\ref{tildeFt}), we can get that $(\tilde{F}_t-f_t)_{z_0}\in \mathcal{O}(K_M)_{z_0}\otimes I(\Psi+\varphi)_{z_0}\subset \mathcal{O}(K_M)_{z_0}\otimes J_{z_0}$, which means that $\xi\cdot \tilde{F}_t=\xi\cdot f_t=1$. Besides, since $v_t(\Psi)\geq \Psi$, we have
	\begin{flalign*}
		\begin{split}
			&\left(\int_{M_T}|\tilde{F}_t-(1-b_{t}(\Psi))f_t|^2e^{-\varphi+v_{t}(\Psi)-\Psi}\right)^{1/2}\\
			\geq&\left(\int_{M_T}|\tilde{F}_t-(1-b_{t}(\Psi))f_t|^2e^{-\varphi}\right)^{1/2}\\
			\geq&\left(\int_{M_T}|\tilde{F}_t|^2e^{-\varphi}\right)^{1/2}-\left(\int_{M_T}|(1-b_{t}(\Psi))f_t|^2e^{-\varphi}\right)^{1/2}\\
			\geq&\left(\int_{M_T}|\tilde{F}_t|^2e^{-\varphi}\right)^{1/2}-\left(\int_{\{\Psi<-t\}}|f_t|^2e^{-\varphi}\right)^{1/2}.
		\end{split}
	\end{flalign*}
	Combining with inequality (\ref{tildeFt}), we have
	\begin{flalign*}
		\begin{split}
			&\int_{M_T}|\tilde{F}_t|^2e^{-\varphi}\\
			\leq&2\int_{M_T}|\tilde{F}_t-(1-b_{t}(\Psi))f_t|^2e^{-\varphi+v_{t}(\Psi)-\Psi}+2\int_{\{\Psi<-t\}}|f_t|^2e^{-\varphi}\\
			\leq&2(Ce^{t+1}+1)\int_{\{\Psi<-t\}}|f_t|^2e^{-\varphi}.
		\end{split}
	\end{flalign*}
	Note that
	\begin{flalign*}
		\begin{split}
			\|f_t\|^2_{\lambda,t}=&\int_{M_T}|f_t|^2e^{-\varphi-\Psi_{\lambda,t}}\\
			=&\int_{\{\Psi<-t\}}|f_t|^2e^{-\varphi}+\int_{\{T>\Psi\geq -t\}}|f_t|^2e^{-\varphi-\lambda(\Psi+t)}\\
			\geq&\int_{\{\Psi<-t\}}|f_t|^2e^{-\varphi}.
		\end{split}
	\end{flalign*}
	Then we have
	\[\int_{M_T}|\tilde{F}_t|^2e^{-\varphi}\leq 2(Ce^{t+1}+1)\|f_t\|_{\lambda,t}^2\leq C_1\frac{e^t}{K^{\varphi}_{\xi,\Psi,\lambda}(t)},\]
	where $C_1$ is a positive constant independent on $t$. In addition, $\xi\cdot \tilde{F}_t=1$ implies that
	\[\int_{M_T}|\tilde{F}_t|^2e^{-\varphi}=\|\tilde{F}_t\|^2_{\lambda,T}\geq (K^{\varphi}_{\xi,\Psi,\lambda}(T))^{-1}.\]
	Then we get that
	\[-\log K^{\varphi}_{\xi,\Psi,\lambda}(t)+t\geq C_2,\ \forall t\in [T,+\infty),\]
	where $C_2:=\log (C_1^{-1}K^{\varphi}_{\xi,\Psi,\lambda}(T))$ is a finite constant. Since $-\log K^{\varphi}_{\xi,\Psi,\lambda}(t)+t$ is concave, we get that $-\log K^{\varphi}_{\xi,\Psi,\lambda}(t)+t$ is increasing with respect to $t\in [T,+\infty)$.
\end{proof}

\section{Proofs of Corollary \ref{L2integral} and Corollary \ref{SOPE}}
In this section, we give the proofs of Corollary \ref{L2integral} and Corollary \ref{SOPE}. Before the proofs, we do some preparations.

Let $\varphi$ be a plurisubharmonic function on $M$, and let $f$ be a holomorphic $(n,0)$ form on $M_{t_0}=\{\Psi<-t_0\}$ for some $t_0\geq T$ such that $f\in A^2(M_{t_0},e^{-\varphi})$. Let $z_0\in M$, and assume that $a_{z_0}^f(\Psi;\varphi)<+\infty$. According to Remark \ref{a>0}, we know that $a_{z_0}^f(\Psi;\varphi)\in (0,+\infty)$.

Let $p>2a_{z_0}^f(\Psi;\varphi)$ and $\lambda>0$. Let $\xi\in A^2(M_{t_0},e^{-\varphi})^*\setminus\{0\}$ satisfying $\xi|_{A^2(M_{t_0},e^{-\varphi})\cap J_p}\equiv 0$, where $J_p:=I(p\Psi+\varphi)_{z_0}$. Denote that
\[K_{\xi,p,\lambda}(t):=\sup_{\tilde{f}\in A^2(M_{t_0},e^{-\varphi})}\frac{|\xi\cdot \tilde{f}|^2}{\|\tilde{f}\|^2_{p,\lambda,t}},\]
where
\[\|\tilde{f}\|_{p,\lambda,t}:=\left(\int_{M_{t_0}}|\tilde{f}|^2e^{-\lambda\max\{p\Psi+t,0\}-\varphi}\right)^{1/2},\]
and $t\in [pt_0,+\infty)$. Note that
\[p\Psi=\min\{p\psi+(2\lceil p\rceil-2p)\log|F|-2\log|F^{\lceil p\rceil}|,-pt_0\}\]
on $\{p\Psi<-pt_0\}$ for any $p>0$, where $\lceil p\rceil :=\min\{m\in\mathbb{Z} : m\geq p\}$. Then definition of $J_p$ shows that $f\in A^2(M_{t_0},e^{-\varphi})\setminus(A^2(M_{t_0},e^{-\varphi})\cap J_p)$, which implies that $A^2(M_{t_0},e^{-\varphi})\cap J_p$ is a proper subspace of $A^2(M_{t_0},e^{-\varphi})$, and $K_{\xi,p,\lambda}(pt_0)\in (0,+\infty)$. Then Theorem \ref{increasing} tells us that $-\log K_{\xi,p,\lambda}(t)+t$ is increasing with respect to $t\in [pt_0,+\infty)$, which implies that
\begin{equation}\label{-logK(t)+t}
	-\log K_{\xi,p,\lambda}(t)+t\geq -\log K_{\xi,p,\lambda}(pt_0)+pt_0, \ \forall t\in [pt_0,+\infty).
\end{equation}

Since $f\in A^2(M_{t_0},e^{-\varphi})$, following from inequality (\ref{-logK(t)+t}), we get that
\[\|f\|^2_{p,\lambda,t}\geq\frac{|\xi\cdot f|^2}{K_{\xi,p,\lambda}(t)}\geq e^{-t+pt_0}\frac{|\xi\cdot f|^2}{K_{\xi,p,\lambda}(pt_0)}, \ \forall t\in [pt_0,+\infty).\]
In addition, since $f\notin A^2(M_{t_0},e^{-\varphi})\cap J_p$, according to Lemma \ref{B=C}, we have
\begin{flalign}\label{f>e^-tC}
	\begin{split}
		\|f\|^2_{p,\lambda,t}&\geq\sup_{\substack{\xi\in A^2(M_{t_0},e^{-\varphi})^*\setminus\{0\}\\ \xi|_{A^2(M_{t_0},e^{-\varphi})\cap J_p}\equiv 0}}e^{-t+pt_0}\frac{|\xi\cdot f|^2}{K_{\xi,p,\lambda}(pt_0)}\\
		&=e^{-t+pt_0}C(p\Psi,\varphi,J_p,f,M_{t_0}), \ \forall t\in [pt_0,+\infty).
	\end{split}
\end{flalign}
Note that for any $t\in [pt_0,+\infty)$,
\begin{equation}\label{f2}
	\|f\|^2_{p,\lambda,t}=\int_{\{p\Psi<-t\}}|f|^2e^{-\varphi}+\int_{\{-pt_0>p\Psi\geq -t\}}|f|^2e^{-\lambda(p\Psi+t)-\varphi}.
\end{equation}
Since for any $\lambda>0$,
\[\int_{\{-pt_0>p\Psi\geq -t\}}|f|^2e^{-\lambda(p\Psi+t)-\varphi}\leq\int_{\{-pt_0>p\Psi\geq -t\}}|f|^2e^{-\varphi}<+\infty,\]
and $\lim_{\lambda\rightarrow+\infty}e^{-\lambda(p\Psi+t)}=0$ on $\{-pt_0>p\Psi\geq -t\}$, according to Lebesgue's dominated convergence theorem, we have
\[\lim_{\lambda\rightarrow+\infty}\int_{\{-pt_0>p\Psi\geq -t\}}|f|^2e^{-\lambda(p\Psi+t)-\varphi}=0.\]
Then equality (\ref{f2}) implies
\begin{equation}
	\lim_{\lambda\rightarrow+\infty}\|f\|^2_{p,\lambda,t}=\int_{\{p\Psi<-t\}}|f|^2e^{-\varphi}, \ \forall t\in [pt_0,+\infty).
\end{equation}
Letting $\lambda\rightarrow+\infty$ in inequality (\ref{f>e^-tC}), we get that for any $t\in [pt_0,+\infty)$,
\begin{equation}\label{f>e^-tC2}
	\int_{\{p\Psi<-t\}}|f|^2e^{-\varphi}\geq e^{-t+pt_0}C(p\Psi,\varphi,J_p,f,M_{t_0}).
\end{equation}

Now we give the proof the Corollary \ref{L2integral}.
\begin{proof}[Proof of Corollary \ref{L2integral}]	
	Note that $J_p\subset I_+(2a_{z_0}^f(\Psi;\varphi)\Psi+\varphi)_{z_0}$ for any $p>2a_{z_0}^f(\Psi;\varphi)$. Then we have
	\[C(p\Psi,\varphi,J_p,f,M_{t_0})\geq C(\Psi,\varphi,I_+(2a_{z_0}^f(\Psi;\varphi)\Psi+\varphi)_{z_0},f,M_{t_0}), \ \forall p>2a_{z_0}^f(\Psi;\varphi).\]
	Since $\int_{M_{t_0}}|f|^2e^{-\varphi}<+\infty$, it follows from Lebesgue's dominated convergence theorem and inequality (\ref{f>e^-tC2}) that
	\begin{flalign}
		\begin{split}
			&\int_{\{2a_{z_0}^f(\Psi;\varphi)\Psi<-t\}}|f|^2e^{-\varphi}\\
			=&\lim_{p\rightarrow 2a_{z_0}^f(\Psi;\varphi)+0}\int_{\{p\Psi<-t\}}|f|^2e^{-\varphi}\\
			\geq&\limsup_{p\rightarrow 2a_{z_0}^f(\Psi;\varphi)+0}e^{-t+pt_0}C(p\Psi,\varphi,J_p,f,M_{t_0})\\
			\geq&e^{-t+2a_{z_0}^f(\Psi;\varphi)t_0}C(\Psi,\varphi,I_+(2a_{z_0}^f(\Psi;\varphi)\Psi+\varphi)_{z_0},f,M_{t_0}),
		\end{split}
	\end{flalign}
	for any $t\in [2a_{z_0}^f(\Psi;\varphi)t_0,+\infty)$. Let $r:=e^{-t/2}$, and we get that Corollary \ref{L2integral} holds.
\end{proof}

In the following we give the proof of Corollary \ref{SOPE}.

\begin{proof}[Proof of Corollary \ref{SOPE}]
	For any $q>2a_{z_0}^f(\Psi;\varphi)$, according to inequality (\ref{f>e^-tC2}), we get that for any $t\in [qt_0,+\infty)$,
	\begin{equation}\label{f>e^-tC2:2}
		\int_{\{q\Psi<-t\}}|f|^2e^{-\varphi}\geq e^{-t+qt_0}C(q\Psi,\varphi,J_q,f,M_{t_0}).
	\end{equation}
	
	It follows from Fubini's Theorem that
	\begin{flalign*}
		\begin{split}
			&\int_{\{\Psi<-t_0\}}|f|^2e^{-\varphi-\Psi}\\
			=&\int_{\{\Psi<-t_0\}}\left(|f|^2e^{-\varphi}\int_0^{e^{-\Psi}}\mathrm{d}s\right)\\
			=&\int_0^{+\infty}\left(\int_{\{\Psi<-t_0\}\cap\{s<e^{-\Psi}\}}|f|^2e^{-\varphi}\right)\mathrm{d}s\\
			=&\int_{-\infty}^{+\infty}\left(\int_{\{q\Psi<-qt\}\cap\{\Psi<-t_0\}}|f|^2e^{-\varphi}\right)e^t\mathrm{d}t.
		\end{split}
	\end{flalign*}
	Inequality (\ref{f>e^-tC2:2}) implies that for any $q>2a_o^f(\Psi;\varphi)$,
	\begin{flalign*}
		\begin{split}
			&\int_{t_0}^{+\infty}\left(\int_{\{q\Psi<-qt\}\cap\{\Psi<-t_0\}}|f|^2e^{-\varphi}\right)e^t\mathrm{d}t\\
			\geq&\int_{t_0}^{+\infty}e^{-qt+qt_0}C(q\Psi,\varphi,J_q,f,M_{t_0})\cdot e^t\mathrm{d}t\\
			=&\frac{1}{q-1}e^{t_0}C(q\Psi,\varphi,J_q,f,M_{t_0}),
		\end{split}
	\end{flalign*}
	and
	\begin{flalign*}
		\begin{split}
			&\int_{-\infty}^{t_0}\left(\int_{\{q\Psi<-qt\}\cap\{\Psi<-t_0\}}|f|^2e^{-\varphi_0}\right)e^t\mathrm{d}t\\
			\geq&\int_{-\infty}^{t_0}C(q\Psi,\varphi,J_q,f,M_{t_0})
			\cdot e^t\mathrm{d}t\\
			=&e^{t_0}C(q\Psi,\varphi,J_q,f,M_{t_0}).
		\end{split}
	\end{flalign*}
	Then we have
	\begin{equation}\label{q/q-1e^t_0}
		\int_{M_{t_0}}|f|^2e^{-\varphi-\Psi}\geq \frac{q}{q-1}e^{t_0}C(q\Psi,\varphi,J_q,f,M_{t_0}).
	\end{equation}
	for any $q>2a_{z_0}^f(\Psi;\varphi)$. Note that $J_q\subset I_+(2a_{z_0}^f(\Psi;\varphi)\Psi+\varphi)_{z_0}$ for any $q>2a_{z_0}^f(\Psi;\varphi)$, which implies
	\[C(q\Psi,\varphi,J_q,f,M_{t_0})\geq C(\Psi,\varphi,I_+(2a_{z_0}^f(\Psi;\varphi)\Psi+\varphi)_{z_0},f,M_{t_0}), \ \forall q>2a_{z_0}^f(\Psi;\varphi).\]
	Then inequality (\ref{q/q-1e^t_0}) induces
	\begin{equation}\label{q/q-1}
		\int_{M_{t_0}}|f|^2e^{-\varphi-\Psi}\geq \frac{q}{q-1}e^{t_0}C(\Psi,\varphi,I_+(2a_{z_0}^f(\Psi;\varphi)\Psi+\varphi)_{z_0},f,M_{t_0}).
	\end{equation}
	Let $q\rightarrow 2a_o^f(\Psi;\varphi_0)+0$, then inequality (\ref{q/q-1}) also holds for $q\geq 2a_o^f(\Psi;\varphi_0)$. Thus if $q>1$ satisfying
	\begin{equation}
		\int_{M_{t_0}}|f|^2e^{-\varphi_0-\Psi}< \frac{q}{q-1}e^{t_0}C(\Psi,\varphi,I_+(2a_{z_0}^f(\Psi;\varphi)\Psi+\varphi)_{z_0},f,M_{t_0}),
	\end{equation}
	we have $q<2a_{z_0}^f(\Psi;\varphi)$, which means that $f_{z_0}\in\mathcal{O}(K_M)_{z_0}\otimes I(q\Psi+\varphi)_{z_0}$. Proof of Corollary \ref{SOPE} is done.
\end{proof}

\section{Appendix}
In the appendix, we show an example for explaining the results we obtain in the present paper.

Let $M=\Delta^2$ be the unit polydisc in $\mathbb{C}^2$ with coordinate $(z,w)$, $\psi=2\log|z|$ a plurisubharmonic function on $\Delta^2$, $F=w^k$ a holomorphic function on $\Delta^2$ for some positive function $k$, and $T=0$. Then we have
\[\Psi(z,w):=\min\{\psi-2\log|F|, 0\}=\min\left\{2\log\left|\frac{z}{w^k}\right|, 0\right\},\]
where $\Psi(z,o):=0$ for any $z\in\Delta$. Then for any $t\in [0,+\infty)$, we have
\[M_t:=\{\Psi<-t\}=\left\{(z,w)\in\Delta\times\Delta^* : |z|<e^{-t/2}|w|^k\right\}.\]
Especially, $M_0=\left\{(z,w)\in\Delta\times\Delta^* : |z|<|w|^k\right\}$. Note that $o$ is a boundary point of $M_0$, and $M_0\cup o$ is a Reinhardt domain. Then any holomorphic function on $M_0$ can be written as its Laurent expansion. Let $f\in\mathcal{O}(M_0)$, and we may write that
\[f(z,w)=\sum_{(j,l)\in\mathbb{Z}^2}c_{j,l}z^jw^l.\]
We can compute that
\[\int_{M_0}|f|^2=\sum_{(j,l)\in\mathbb{Z}^2}|c_{j,l}|^2\int_{M_0}|z|^{2j}|w|^{2l},\]
where
\[\int_{M_0}|z|^{2j}|w|^{2l}=\frac{\pi^2}{(j+1)(jk+k+l+1)},\]
if $j\geq 0$, $jk+k+l\geq 0$, otherwise it equals $+\infty$. Similarly, we can compute that
\[\int_{M_0}|f|^2e^{-\Psi}=\sum_{(j,l)\in\mathbb{Z}^2}|c_{j,l}|^2\int_{M_0}|z|^{2j-2}|w|^{2l+2k},\]
where
\[\int_{M_0}|z|^{2j-2}|w|^{2l+2k}=\frac{\pi^2}{j(jk+k+l+1)},\]
if $j\geq 1$, $jk+k+l\geq 0$, otherwise it equals $+\infty$. Then it can be noted that
\[I(\Psi)_o=\{h(z,w)=\sum_{(j,l)\in\mathbb{Z}^2}a_{j,l}z^jw^l \text{\ near \ } o : a_{j,l}=0, \ \forall (j,l)\in \mathcal{S} \},\]
where $\mathcal{S}:=\{(j,l)\in\mathbb{Z}^2 : j=0$ or $l\leq -(jk+k+1)\}$. Let $\xi\in A^2(M_0)^*$ which is defined as:
\begin{equation}\label{xi=c0-k}
	\xi\cdot f=c_{0,-k}, \ \forall f=\sum_{(j,l)\in\mathbb{Z}^2}c_{j,l}z^jw^l \in A^2(M_0),\end{equation}
where it is easy to see that $\xi$ is linear, continuous, $\xi\neq 0$, and $\xi|_{A^2(M_0)\cap I(\Psi)_o}\equiv 0$.

Let $\varphi\equiv 0$, and $\lambda>1$. Then
\[\Psi_{\lambda,t}:=\lambda\max\{\Psi+t,0\}=\lambda\max\left\{2\log\frac{|z|}{|w|^k}+t,0\right\}.\]
Now we can calculate that for any $t\in [0,+\infty)$, 
\begin{equation*}
	\begin{split}
		K_{\xi,\Psi,\lambda}^{0}(t)&=\left(\int_{M_0}|w^{-k}|^2e^{-\Psi_{\lambda,t}}\right)^{-1}\\
		&=\frac{1}{\pi^2}\cdot\frac{\lambda-1}{\lambda e^{-t}-e^{-\lambda t}},
	\end{split}
\end{equation*}
for $\lambda\neq 1$, and
\[K_{\xi,\Psi,1}^{0}(t)=\frac{1}{\pi^2}e^t,\] 
for $\lambda=1$. Then it can be checked that for any $\lambda>0$, $\log K_{\xi,\Psi,\lambda}^{0}(t)$ is convex with respect to $t\in[0,+\infty)$, and $-\log K_{\xi,\Psi,\lambda}^{0}(t)+t$ increasing on $[0,+\infty)$. These are compatible to Theorem \ref{concavity} and Theorem \ref{increasing}.

Let $g:=w^{-k}\in A^2(M_0)$. Then it can be verified that
\[C(\Psi,0,I_+(\Psi)_o,g,M_0)=\|g\|^2_0=\int_{M_0}|g|^2=\pi^2,\]
where exactly we have
\[C(\Psi,0,I_+(\Psi)_o,g,M_0)=\frac{|\xi\cdot g|^2}{K_{\xi,\Psi,\lambda}^{0}(0)}\]
for the functional $\xi$ described in (\ref{xi=c0-k}), which can be referred to Lemma \ref{B=C}. In addition, it can be computed that $a^g_o(\Psi;0)=\frac{1}{2}$,
and
\[\frac{1}{r^2}\int_{\{\frac{1}{2}\Psi<\log r\}}|g|^2=\pi^2=C(\Psi,0,I_+(\Psi)_o,g,M_0)\]
for any $r\in (0,1]$. This is compatible to Corollary \ref{L2integral}. Similar computations can also be given to Corollary \ref{SOPE} under these assumptions.

\vspace{.1in} {\em Acknowledgements}. We would like to thank Zhitong Mi and Zheng Yuan for checking this paper. The second named author was supported by National Key R\&D Program of China 2021YFA1003100, NSFC-11825101, NSFC-11522101 and NSFC-11431013.

\bibliographystyle{references}
\bibliography{xbib}

\end{document}